\documentclass[10pt,a4paper]{article}

\usepackage[margin=1in]{geometry}
\usepackage[T1]{fontenc}
\usepackage[utf8]{inputenc}
\usepackage{lmodern}
\usepackage{amsmath,amssymb,amsfonts,amsthm,mathtools}
\usepackage{enumitem}
\usepackage[hidelinks]{hyperref}

\newtheorem{thm}{Theorem}[section]
\newtheorem{lem}[thm]{Lemma}
\newtheorem{prop}[thm]{Proposition}
\newtheorem{cor}[thm]{Corollary}
\newtheorem{defi}[thm]{Definition}
\newtheorem{fac}[thm]{Fact}
\newtheorem{remark}[thm]{Remark}
\newtheorem{exam}[thm]{Example}

\newcommand{\Ker}{\operatorname{Ker}}
\newcommand{\Gal}{\operatorname{Gal}}
\newcommand{\Tr}{\operatorname{Tr}}
\newcommand{\Norm}{\operatorname{N}}
\newcommand{\Res}{\operatorname{Res}}
\newcommand{\one}{\mathbf{1}}
\newcommand{\rank}{\operatorname{rank}}

\numberwithin{equation}{section}

\renewcommand{\thefootnote}{\fnsymbol{footnote}}

\title{Cyclotomic Numbers of Order $q-1$ over $\mathbb{F}_{q^r}$}
\author{Hayaki Kudo\footnotemark[1], Yuto Nogata\footnotemark[2]\\
\normalsize{3, Bunkyo-cho, Hirosaki-shi, Aomori, 036-8561, JAPAN}}
\date{\today}

\begin{document}

\maketitle
\footnotetext[1]{Graduate School of Science and Technology, Hirosaki University. Email: h25ms109@hirosaki-u.ac.jp}
\footnotetext[2]{Doctoral Course, Safety Science and Technology, Safety System Engineering, Graduate School of Science and Technology, Hirosaki University. Email: h26ds201@hirosaki-u.ac.jp}
\setcounter{footnote}{0}
\renewcommand{\thefootnote}{\arabic{footnote}}

\section*{Abstract}
Let $q=p^n$, $r\in \mathbb{Z}_{\ge 2}$, $e=q-1$, and $k=\frac{q^r-1}{e}$. In this paper, we study the cyclotomic numbers $(a,b)_{q-1}$ over $\mathbb{F}_{q^r}$. We prove that $(a,b)_{q-1}\le \left\lceil \frac{k}{2}\right\rceil$ for all $0\le a,b\le q-2$ except when $q=2$ and $r\ge 3$. We also give sharper bounds for prime values of $r$, especially for $r=2$ and $r=3$.

\noindent\textbf{MSC2020:} 11T22, 11T24, 12E20

\noindent\textbf{Keywords:} Cyclotomy, cyclotomic number, finite field

\section{Introduction and main results}

We begin by recalling the standard definitions of cyclotomic cosets, cyclotomic numbers, and the matrices $C^{(a,b)}$.

\begin{defi}[\cite{Storer1967}]
Assume that $k\mid |\mathbb{F}|-1$, and let $e:=\frac{|\mathbb{F}|-1}{k}$. Fix $\omega$ with $\langle\omega\rangle=\mathbb{F}^{\times}$. For $0\le a\le e-1$, define the cyclotomic coset $C_a:=\omega^a\langle \omega^e\rangle$. For $0\le a,b\le e-1$, define the cyclotomic number by $(a,b)_e:=\#\{x\in C_a\mid x+1\in C_b\}$.
\end{defi}

\begin{defi}
For $0\le a,b\le e-1$, define the $k\times k$ matrix $C^{(a,b)}=(c_{ij}^{(a,b)})_{0\le i,j\le k-1}$ over $\mathbb{F}$ by
\[
  c_{ij}^{(a,b)}=
  \begin{cases}
    1+\omega^{ak}-\omega^{bk}, & \textup{if } i=j,\\
    \binom{k}{j-i}, & \textup{if } i<j,\\
    \omega^{ak}\binom{k}{i-j}, & \textup{if } i>j.
  \end{cases}
\]
\end{defi}
According to \cite[Lemma 3.1]{BetsumiyaHirasakaKomatsuMunemasa2013}, we have $(a,b)_e=k-\rank_{\mathbb{F}}(C^{(a,b)})$. The cyclotomic numbers of small order have been studied extensively in the classical literature \cite{BaumertFredricksen1967,EvansHill1979,Lehmer1954,LehmerVandiver1957,Whiteman1957,Wilson1972}. More precisely, the cyclotomic numbers $(a,b)_e$ are explicitly determined for $e\le 12$ and for $e=14,15,16,18,20,24$; see \cite[p.~152]{BerndtEvansWilliams1998}.  \cite{Katre1989} proved that, for fixed $e$ and $|\mathbb{F}|\to\infty$, we have $(a,b)_e\sim \frac{|\mathbb{F}|}{e^2}$ for all $a,b\in \mathbb{Z}$. On the other hand, \cite{BetsumiyaHirasakaKomatsuMunemasa2013,DoDucLeungSchmidt2020} established upper bounds for cyclotomic numbers for large characteristic. In particular, \cite[Theorem 1.1(i)--(v)]{BetsumiyaHirasakaKomatsuMunemasa2013} proves the following.

\begin{fac}[\cite{BetsumiyaHirasakaKomatsuMunemasa2013}, Theorem 1.1(i)--(v)]
Let $q$ be a power of an odd prime $p^{n}$. Then the following hold.
\begin{enumerate}
\item[\textup{(i)}] If $p>\frac{3k}{2}-1$, then $(a,b)_{e}\le \lceil k/2\rceil$ for all $0\le a,b\le e-1$.
\item[\textup{(ii)}] If $k$ is odd and $p>\frac{3k}{2}$, then $(a,a)_{e}\le \lceil k/2\rceil-1$ for all $0\le a\le e-1$.
\item[\textup{(iii)}] If $p>\frac{3k}{2}$, then $(0,0)_{e}\le \lceil k/2\rceil-1$.
\item[\textup{(iv)}] If $6\mid k$ and $p$ is sufficiently large compared with $k$, then $(0,0)_{e}=2$.
\item[\textup{(v)}] If $6\nmid k$ and $p$ is sufficiently large compared with $k$, then $(0,0)_{e}=0$.
\end{enumerate}
\end{fac}

In this paper, we take $q=p^n$, $r\in \mathbb{Z}_{\ge 2}$, and $k=\frac{q^r-1}{q-1}$; then $e=q-1$. We study the cyclotomic numbers $(a,b)_{q-1}$ for $0\le a,b\le q-2$. Since $q$ is fixed and $e=q-1$ is also fixed, \cite{Katre1989} yields $(a,b)_{q-1}\sim \frac{q^r}{(q-1)^2}$ as $r\to \infty$. The point of this paper is that, except for the cases $q=2$ and $r\ge 3$, the inequality $(a,b)_{q-1}\le \left\lceil \frac{k}{2}\right\rceil$ holds in this family. We prove the following theorem.

\begin{thm}[Main theorem]\label{thm:maintheorem}
Let $q=p^n$, $r\in \mathbb{Z}_{\ge 2}$, $k=\frac{q^r-1}{q-1}$, and $e=q-1$. Then the following statements hold for every $0\le a,b\le q-2$.
\begin{enumerate}[label=\textup{(\arabic*)},leftmargin=2.5em]
\item If $q\neq 2$, then $(a,b)_{q-1}\le \left\lceil \frac{k}{2}\right\rceil$. If $q=2$, then $(0,0)_1=2^r-2$. Hence equality holds when $r=2$, and the inequality fails when $r\ge 3$.

\item With the notation of Definition \ref{def:Wuv-cayley}, we have $(a,b)_{q-1}=\frac{q^r+(-1)^r(q-2)^r-2}{(q-1)^2}+(-1)^{r-1}W_{(-1)^r\omega^{ak},\omega^{bk}}(q,r)$.

\item If $r$ is an odd prime $\ell$, then
\[
0\le (a,b)_{q-1}\le
\begin{cases}
\ell q^{\ell-2}+1, & \textup{if } p=\ell,\\
\ell q^{\ell-2}+\ell-1, & \textup{if } p\neq\ell.
\end{cases}
\]
If $r=2$, then $(a,b)_{q-1}\in \{0,1,2\}$. If $r=3$, then $6\le (a,b)_{q-1}\le 2q+4$.
\end{enumerate}
\end{thm}

\begin{proof}
\textup{(1)} follows from Corollary \ref{cor:character-sum-half-k-general}. \textup{(2)} follows from Proposition \ref{prop:general-order-qminus1-graph}.
\textup{(3)} follows from Corollaries \ref{cor:r2-general-bound}, \ref{cor:V-bound-odd-prime-general}, and \ref{cor:r3-structure-general}.
\end{proof}

In \S\ref{sec:lemmas}, we collect the facts and lemmas needed for the proof. In \S\ref{sec:order-qminus1}, we investigate whether $(a,b)_{q-1}\le \left\lceil \frac{k}{2}\right\rceil$ holds or not. In \S\ref{sec:cayley-digraph}, we give an exact formula in terms of a directed Cayley graph. In \S\ref{sec:small-r}, we prove sharper bounds in the case where $r$ is an odd prime $\ell$, with the case $r=3$ considered separately.

\section{Notation and preliminaries}\label{sec:lemmas}

In this section, we collect the facts and lemmas needed for the proof of the main theorem. The value of $\lceil k/2\rceil$ is 
\begin{equation}\label{eq:ceil-half-k}
\left\lceil \frac{k}{2}\right\rceil=
\begin{cases}
\frac{q^r+q-2}{2(q-1)} & \text{if $q$ is even,} \\
\frac{q^r+q-2}{2(q-1)} & \text{if $q$ is odd and $r$ is odd,} \\
\frac{q^r-1}{2(q-1)} & \text{if $q$ is odd and $r$ is even.}
\end{cases}
\end{equation}

Throughout this section, $\mathbb{F}$ denotes a finite field of characteristic $p$. A multiplicative character of $\mathbb{F}^{\times}$ means a group homomorphism $\chi:\mathbb{F}^{\times}\to\mathbb{C}^{\times}$. We extend each multiplicative character to all of $\mathbb{F}$ by setting $\chi(0):=0$. Let $\widehat{\mathbb{F}^{\times}}$ denote the group of all multiplicative characters of $\mathbb{F}^{\times}$. The trivial multiplicative character is denoted by $\one$. For a condition $P$, we write $\one_P:=1$ if $P$ holds, otherwise $\one_P:=0$.

\begin{fac}[c.f.\cite{LidlNiederreiter1997}, Theorem~5.4]\label{fac:standard-facts}
Let $\mathbb{F}$ be a finite field. For every $u\in \mathbb{F}$, we have $\frac{1}{|\mathbb{F}|-1}\sum_{\chi\in\widehat{\mathbb{F}^{\times}}}\chi(u)=\one_{u=1}$.
\end{fac}

\begin{lem}\label{lem:nontrivial-multiplicative-character-sum-zero}
Let $\eta\in\widehat{\mathbb{F}^{\times}}\backslash\{\one\}$. Then
$\sum_{y\in\mathbb{F}}\eta(y)=0$.
\end{lem}

\begin{proof}
Since $\eta(0)=0$, we have
$\sum_{y\in\mathbb{F}}\eta(y)=\sum_{y\in\mathbb{F}^{\times}}\eta(y)$.
Let $\langle \omega\rangle= \mathbb{F}^{\times}$. Since $\eta\neq\one$, we have $\eta(\omega)\neq 1$. Hence
$\sum_{y\in\mathbb{F}^{\times}}\eta(y)=\sum_{s=0}^{|\mathbb{F}|-2}\eta(\omega^s)=\sum_{s=0}^{|\mathbb{F}|-2}\eta(\omega)^s
=\frac{\eta(\omega)^{|\mathbb{F}|-1}-1}{\eta(\omega)-1}=0$.
Therefore $\sum_{y\in\mathbb{F}}\eta(y)=0$.
\end{proof}

Now let $\mathbb{F}/\mathbb{K}$ be a finite extension of finite fields, and put $d:=[\mathbb{F}:\mathbb{K}]$. We define $\sigma_{\mathbb{F}/\mathbb{K}}:\mathbb{F}\to\mathbb{F}$ $(\sigma_{\mathbb{F}/\mathbb{K}}(x)\mapsto x^{|\mathbb{K}|})$. Then $\Gal(\mathbb{F}/\mathbb{K})=\langle\sigma_{\mathbb{F}/\mathbb{K}}\rangle$ and we define the following notation.

\begin{itemize}
  \item \begin{tabular}[t]{@{}p{0.28\linewidth}@{\ }c@{\ }l@{}}
    the relative trace
    & ; &
    \begin{tabular}[t]{@{}p{0.1\linewidth}@{\ }c@{\ }p{0.1\linewidth}@{\ }l@{}}
      \hfill$\Tr_{\mathbb{F}/\mathbb{K}}$ & : & $\mathbb{F}\to\mathbb{K}$ & $\bigl(x\mapsto \sum_{\mu=0}^{d-1}\sigma_{\mathbb{F}/\mathbb{K}}^{\mu}(x)\bigr)$
    \end{tabular}
  \end{tabular}

  \item \begin{tabular}[t]{@{}p{0.28\linewidth}@{\ }c@{\ }l@{}}
    the relative norm
    & ; &
    \begin{tabular}[t]{@{}p{0.1\linewidth}@{\ }c@{\ }p{0.1\linewidth}@{\ }l@{}}
      \hfill$\Norm_{\mathbb{F}/\mathbb{K}}$ & : & $\mathbb{F}\to\mathbb{K}$ & $\bigl(x\mapsto \prod_{\mu=0}^{d-1}\sigma_{\mathbb{F}/\mathbb{K}}^{\mu}(x)\bigr)$
    \end{tabular}
  \end{tabular}

  \item \begin{tabular}[t]{@{}p{0.28\linewidth}@{\ }c@{\ }l@{}}
    the additive character of $\mathbb{F}$
    & ; &
    \begin{tabular}[t]{@{}p{0.1\linewidth}@{\ }c@{\ }p{0.1\linewidth}@{\ }l@{}}
      \hfill$\mathbf{e}_{\mathbb{F}}$ & : & $\mathbb{F}\to\mathbb{C}^{\times}$ & $\bigl(x\mapsto \exp \bigl(\frac{2\pi i}{p}\Tr_{\mathbb{F}/\mathbb{F}_p}(x)\bigr)\bigr)$
    \end{tabular}
  \end{tabular}

  \item \begin{tabular}[t]{@{}p{0.28\linewidth}@{\ }c@{\ }l@{}}
    the Gauss sum
    & ; &
    \begin{tabular}[t]{@{}p{0.1\linewidth}@{\ }c@{\ }p{0.1\linewidth}@{\ }l@{}}
      \hfill$G_{\mathbb{F}}$ & : & $\widehat{\mathbb{F}^{\times}}\to\mathbb{C}$ & $\bigl(\chi\mapsto \sum_{x\in \mathbb{F}^{\times}}\chi(x)\mathbf{e}_{\mathbb{F}}(x)\bigr)$
    \end{tabular}
  \end{tabular}

  \item \begin{tabular}[t]{@{}p{0.28\linewidth}@{\ }c@{\ }l@{}}
    the Jacobi sums
    & ; &
    $\begin{aligned}[t]
      J'_{\mathbb{F}},J_{\mathbb{F}}
      &: \widehat{\mathbb{F}^{\times}}\times\widehat{\mathbb{F}^{\times}}\to\mathbb{C}\ \text{by}\ 
      \left\{
      \begin{aligned}
        J'_{\mathbb{F}}(\chi,\psi)&:=\sum_{x\in \mathbb{F}}\chi(x)\psi(1-x),\\
        J_{\mathbb{F}}(\chi,\psi)&:=\sum_{x\in \mathbb{F}}\chi(x)\psi(x+1)
      \end{aligned}
      \right.
    \end{aligned}$
  \end{tabular}
\end{itemize}

The restriction of $\Norm_{\mathbb{F}/\mathbb{K}}$ to $\mathbb{F}^{\times}$ is a surjective group homomorphism onto $\mathbb{K}^{\times}$. Under our assumption, $\Norm_{\mathbb{F}_{q^r}/\mathbb{F}_q}(x)=x^k$ for $x\in\mathbb{F}_{q^r}$. By the definition, we have
\begin{equation}\label{eq:J-Jprime-relation}
J_{\mathbb{F}}(\chi,\psi)=\chi(-1)J'_{\mathbb{F}}(\chi,\psi).
\end{equation}

\begin{fac}[\cite{BerndtEvansWilliams1998}, Theorems~1.1.4(a,d), 2.1.3(a), 11.5.2]\label{fac:gauss-jacobi}
The following statements hold.
\begin{enumerate}[label=\textup{(\arabic*)},leftmargin=2.5em]
  \item If $\chi,\psi,\chi\psi\neq \one$ in $\widehat{\mathbb{F}^{\times}}$, then $J'_{\mathbb{F}}(\chi,\psi)=\frac{G_{\mathbb{F}}(\chi)G_{\mathbb{F}}(\psi)}{G_{\mathbb{F}}(\chi\psi)}$.

  \item If $\rho\neq \one$ in $\widehat{\mathbb{F}^{\times}}$, then $|G_{\mathbb{F}}(\rho)|=|\mathbb{F}|^{1/2}$.

  \item If $\rho\in\widehat{\mathbb{K}^{\times}}$ and $\widetilde{\rho}:=\rho\circ\Norm_{\mathbb{F}/\mathbb{K}}$, then $G_{\mathbb{F}}(\widetilde{\rho})=(-1)^{[\mathbb{F}:\mathbb{K}]-1}G_{\mathbb{K}}(\rho)^{[\mathbb{F}:\mathbb{K}]}$.
\end{enumerate}
\end{fac}

\begin{lem}\label{lem:jacobi-absolute-value}
Let $q=p^n$, and let $\chi,\psi\in\widehat{\mathbb{F}_q^\times}$. If $\chi,\psi,\chi\psi\neq\one$, then
\begin{equation}\label{eq:jacobi-absolute-value}
|J_{\mathbb{F}_q}(\chi,\psi)|=q^{1/2}.
\end{equation}
\end{lem}

\begin{proof}
By Fact \ref{fac:gauss-jacobi} \textup{(1)} applied to $\mathbb{F}_q$, we have
$J'_{\mathbb{F}_q}(\chi,\psi)=\frac{G_{\mathbb{F}_q}(\chi)G_{\mathbb{F}_q}(\psi)}{G_{\mathbb{F}_q}(\chi\psi)}$.
Combining this with Fact \ref{fac:gauss-jacobi} \textup{(2)} applied to $\mathbb{F}_q$, we obtain
$|J'_{\mathbb{F}_q}(\chi,\psi)|=q^{1/2}$.
By \eqref{eq:J-Jprime-relation}, we have
$|J_{\mathbb{F}_q}(\chi,\psi)|=|J'_{\mathbb{F}_q}(\chi,\psi)|=q^{1/2}$.
\end{proof}

\begin{lem}\label{lem:jacobi-special-values}
Let $\chi,\psi\in\widehat{\mathbb{F}^{\times}}$. Then the following statements hold.
\begin{enumerate}[label=\textup{(\arabic*)},leftmargin=2.5em]
\item If $\chi=\one$ and $\psi=\one$, then
\begin{equation}\label{eq:jacobi-special-trivial-trivial}
J'_{\mathbb{F}}(\one,\one)=J_{\mathbb{F}}(\one,\one)=|\mathbb{F}|-2.
\end{equation}

\item If $\chi=\one$ and $\psi\neq\one$, then
\begin{equation}\label{eq:jacobi-special-left-trivial}
J'_{\mathbb{F}}(\one,\psi)=J_{\mathbb{F}}(\one,\psi)=-1.
\end{equation}

\item If $\chi\neq\one$ and $\psi=\one$, then
\begin{subequations}
\begin{align}
J'_{\mathbb{F}}(\chi,\one)&=-1,\label{eq:jacobi-special-right-trivial-Jprime}\\
J_{\mathbb{F}}(\chi,\one)&=-\chi(-1).\label{eq:jacobi-special-right-trivial-J}
\end{align}
\end{subequations}

\item If $\chi\neq\one$, $\psi\neq\one$, and $\chi\psi=\one$, then
\begin{subequations}
\begin{align}
J'_{\mathbb{F}}(\chi,\psi)&=-\chi(-1),\label{eq:jacobi-special-inverse-Jprime}\\
J_{\mathbb{F}}(\chi,\psi)&=-1.\label{eq:jacobi-special-inverse-J}
\end{align}
\end{subequations}
\end{enumerate}
\end{lem}

\begin{proof}
First, we compute the values of $J'_{\mathbb{F}}$. Then the corresponding values of $J_{\mathbb{F}}$ follow from \eqref{eq:J-Jprime-relation}.

\noindent\textup{(1)}: Assume that $\chi=\one$ and $\psi=\one$. Then
$J'_{\mathbb{F}}(\one,\one)=\sum_{x\in\mathbb{F}}\one(x)\one(1-x)=|\mathbb{F}\setminus\{0,1\}|=|\mathbb{F}|-2$.
Since $\one(-1)=1$, \eqref{eq:J-Jprime-relation} gives
$J_{\mathbb{F}}(\one,\one)=|\mathbb{F}|-2$.
It proves \eqref{eq:jacobi-special-trivial-trivial}.

\noindent\textup{(2)}: Assume that $\chi=\one$ and $\psi\neq\one$. Then $J'_{\mathbb{F}}(\one,\psi)=\sum_{x\in\mathbb{F}\setminus\{0\}}\psi(1-x)$. Putting $y=1-x$, we obtain $J'_{\mathbb{F}}(\one,\psi)=\sum_{y\in\mathbb{F}\setminus\{1\}}\psi(y)=-1$ by Lemma \ref{lem:nontrivial-multiplicative-character-sum-zero}. Since $\one(-1)=1$, \eqref{eq:J-Jprime-relation} gives $J_{\mathbb{F}}(\one,\psi)=-1$. It proves \eqref{eq:jacobi-special-left-trivial}.

\noindent\textup{(3)}: Assume that $\chi\neq\one$ and $\psi=\one$. Then $J'_{\mathbb{F}}(\chi,\one)=\sum_{x\in\mathbb{F}\setminus\{1\}}\chi(x)=-1$ by Lemma \ref{lem:nontrivial-multiplicative-character-sum-zero}. Hence \eqref{eq:J-Jprime-relation} gives $J_{\mathbb{F}}(\chi,\one)=-\chi(-1)$. It proves \eqref{eq:jacobi-special-right-trivial-Jprime} and \eqref{eq:jacobi-special-right-trivial-J}.

\noindent\textup{(4)}: Assume that $\chi\neq\one$, $\psi\neq\one$, and $\chi\psi=\one$. Then $\psi=\chi^{-1}$. Hence $J'_{\mathbb{F}}(\chi,\psi)=\sum_{x\in\mathbb{F}\setminus\{0,1\}}\chi(x(1-x)^{-1})$. The map $\mathbb{F}\setminus\{0,1\}\to \mathbb{F}\setminus\{0,-1\}$ $(x\mapsto x(1-x)^{-1})$ is a bijection. Therefore $J'_{\mathbb{F}}(\chi,\psi)=\sum_{y\in\mathbb{F}\setminus\{0,-1\}}\chi(y)=-\chi(-1)$ by Lemma \ref{lem:nontrivial-multiplicative-character-sum-zero}. By \eqref{eq:J-Jprime-relation}, we get $J_{\mathbb{F}}(\chi,\psi)=\chi(-1)(-\chi(-1))=-1$. It proves \eqref{eq:jacobi-special-inverse-Jprime} and \eqref{eq:jacobi-special-inverse-J}.
\end{proof}

\begin{lem}\label{lem:jacobi-norm-comparison}
Let $q=p^n$, $r\in \mathbb{Z}_{\ge 2}$, and $\chi,\psi\in\widehat{\mathbb{F}_q^\times}$. Put $\widetilde{\chi}:=\chi\circ\Norm_{\mathbb{F}_{q^r}/\mathbb{F}_q}$ and $\widetilde{\psi}:=\psi\circ\Norm_{\mathbb{F}_{q^r}/\mathbb{F}_q}$. Then the following hold.
\begin{equation}\label{eq:jacobi-norm-comparison-trivial}
J_{\mathbb{F}_{q^r}}(\widetilde{\chi},\widetilde{\psi})=q^r-2,
\qquad
J_{\mathbb{F}_q}(\chi,\psi)^r=(q-2)^r
\quad\textup{if }(\chi,\psi)=(\one,\one).
\end{equation}
\begin{equation}\label{eq:jacobi-norm-comparison-nontrivial}
J_{\mathbb{F}_{q^r}}(\widetilde{\chi},\widetilde{\psi})
=
(-1)^{r-1}J_{\mathbb{F}_q}(\chi,\psi)^r
\quad\textup{if }(\chi,\psi)\neq(\one,\one).
\end{equation}
\end{lem}

\begin{proof}
First assume that $(\chi,\psi)=(\one,\one)$. Then $\widetilde{\chi}=\one$ and $\widetilde{\psi}=\one$. By \eqref{eq:jacobi-special-trivial-trivial}, we have
$J_{\mathbb{F}_{q^r}}(\widetilde{\chi},\widetilde{\psi})=q^r-2$
and
$J_{\mathbb{F}_q}(\chi,\psi)=q-2$.
Hence
$J_{\mathbb{F}_q}(\chi,\psi)^r=(q-2)^r$.
It proves \eqref{eq:jacobi-norm-comparison-trivial}. Next assume that $(\chi,\psi)\neq(\one,\one)$.

\noindent\textup{(i)}: If $\chi, \psi, \chi\psi\neq\one$, then Fact \ref{fac:gauss-jacobi} \textup{(1)} and \textup{(3)} give
$J'_{\mathbb{F}_{q^r}}(\widetilde{\chi},\widetilde{\psi})=(-1)^{r-1}J'_{\mathbb{F}_q}(\chi,\psi)^r$.
By \eqref{eq:J-Jprime-relation}, we have
$J_{\mathbb{F}_{q^r}}(\widetilde{\chi},\widetilde{\psi})=\widetilde{\chi}(-1)J'_{\mathbb{F}_{q^r}}(\widetilde{\chi},\widetilde{\psi})$.
Since $\widetilde{\chi}(-1)=\chi(\Norm_{\mathbb{F}_{q^r}/\mathbb{F}_q}(-1))=\chi((-1)^r)=\chi(-1)^r$, we obtain
$J_{\mathbb{F}_{q^r}}(\widetilde{\chi},\widetilde{\psi})=(-1)^{r-1}J_{\mathbb{F}_q}(\chi,\psi)^r$.

\noindent\textup{(ii)}: If $\chi=\one$ and $\psi\neq\one$, then \eqref{eq:jacobi-special-left-trivial} gives
$J_{\mathbb{F}_{q^r}}(\widetilde{\chi},\widetilde{\psi})=-1$
and
$J_{\mathbb{F}_q}(\chi,\psi)=-1$.
Thus
$J_{\mathbb{F}_{q^r}}(\widetilde{\chi},\widetilde{\psi})=(-1)^{r-1}J_{\mathbb{F}_q}(\chi,\psi)^r$.

\noindent\textup{(iii)}: If $\chi\neq\one$ and $\psi=\one$, then \eqref{eq:jacobi-special-right-trivial-J} gives
$J_{\mathbb{F}_{q^r}}(\widetilde{\chi},\widetilde{\psi})=-\widetilde{\chi}(-1)=-\chi((-1)^r)$
and
$J_{\mathbb{F}_q}(\chi,\psi)=-\chi(-1)$.
Hence
$J_{\mathbb{F}_{q^r}}(\widetilde{\chi},\widetilde{\psi})=(-1)^{r-1}J_{\mathbb{F}_q}(\chi,\psi)^r$.

\noindent\textup{(iv)}: If $\chi\neq\one$, $\psi\neq\one$, and $\chi\psi=\one$, then \eqref{eq:jacobi-special-inverse-J} gives
$J_{\mathbb{F}_{q^r}}(\widetilde{\chi},\widetilde{\psi})=-1$
and
$J_{\mathbb{F}_q}(\chi,\psi)=-1$.
Thus
$J_{\mathbb{F}_{q^r}}(\widetilde{\chi},\widetilde{\psi})=(-1)^{r-1}J_{\mathbb{F}_q}(\chi,\psi)^r$.
It proves \eqref{eq:jacobi-norm-comparison-nontrivial}.
\end{proof}

\begin{lem}\label{lem:weighted-jacobi-norm-comparison}
Let $q=p^n$, $r\in \mathbb{Z}_{\ge 2}$, and $u, v\in\mathbb{F}_q^\times$. Put $\widetilde{\chi}:=\chi\circ\Norm_{\mathbb{F}_{q^r}/\mathbb{F}_q}$ and $\widetilde{\psi}:=\psi\circ\Norm_{\mathbb{F}_{q^r}/\mathbb{F}_q}$. Then
\begin{equation}\label{eq:weighted-jacobi-norm-comparison}
\sum_{\chi,\psi\in\widehat{\mathbb{F}_q^\times}}\chi(u^{-1})\psi(v^{-1})J_{\mathbb{F}_{q^r}}(\widetilde{\chi},\widetilde{\psi})=q^r-2+(-1)^{r-1}\left(\sum_{\chi,\psi\in\widehat{\mathbb{F}_q^\times}}\chi(u^{-1})\psi(v^{-1})J_{\mathbb{F}_q}(\chi,\psi)^r-(q-2)^r\right).
\end{equation}
\end{lem}

\begin{proof}
By \eqref{eq:jacobi-norm-comparison-nontrivial}, we have
\[
\sum_{(\chi,\psi)\neq(\one,\one)}
\chi(u^{-1})\psi(v^{-1})J_{\mathbb{F}_{q^r}}(\widetilde{\chi},\widetilde{\psi})
=
(-1)^{r-1}
\sum_{(\chi,\psi)\neq(\one,\one)}
\chi(u^{-1})\psi(v^{-1})J_{\mathbb{F}_q}(\chi,\psi)^r.
\]
If $(\chi,\psi)=(\one,\one)$, then $\chi(u^{-1})\psi(v^{-1})=1$, and this case is exactly described by \eqref{eq:jacobi-norm-comparison-trivial}. Hence \eqref{eq:weighted-jacobi-norm-comparison} follows.
\end{proof}

Let $\bar{\mathbb{F}}$ denote the algebraic closure of $\mathbb{F}$.
Let $f(X)=a_mX^m+a_{m-1}X^{m-1}+\cdots+a_0$ and
$g(X)=b_nX^n+b_{n-1}X^{n-1}+\cdots+b_0$.
Denote $a_t=0$ for $t\notin\{0,1,\dotsc,m\}$ and $b_t=0$ for $t\notin\{0,1,\dotsc,n\}$.
Define the Sylvester matrix $S(f,g)$ by
$S(f,g):=\begin{psmallmatrix} A_{11} & A_{12}\\ A_{21} & A_{22}\end{psmallmatrix}$.
Here $A_{11}:=(a_{m+i-j})_{1\le i,j\le n}$, $A_{12}:=(a_{m-n+i-j})_{1\le i\le n,\ 1\le j\le m}$, $A_{21}:=(b_{n+i-j})_{1\le i\le m,\ 1\le j\le n}$, and $A_{22}:=(b_{i-j})_{1\le i,j\le m}$.
Also define $\Res(f,g):=\det S(f,g)$.

\begin{fac}[\cite{Lang2002}, Chapter~IV, \S 8, Corollary~8.4]\label{fac:resultant-common-root}
Let $f(X), g(X)\in \mathbb{F}_q[X]$ be nonzero polynomials. Then $\Res(f,g)\neq 0$ if and only if $f$ and $g$ have no common root in the algebraic closure $\bar{\mathbb{F}}_q$.
\end{fac}

\begin{fac}[c.f.\cite{Lang2002}, Chapter~VI, \S 6, Theorem~6.4\textup{(ii)}]\label{fac:artin-schreier-quadratic}
Let $\mathbb{K}$ be a finite field of characteristic $2$, and $c\in\mathbb{K}$. Then $X^2+X+c$ is irreducible over $\mathbb{K}$ if and only if $\Tr_{\mathbb{K}/\mathbb{F}_2}(c)=1$.
\end{fac}

\begin{lem}\label{lem:minpoly-norm}
Assume that $x\in\bar{\mathbb{F}}_q$ has degree $d$ over $\mathbb{F}_q$. Let 
$m_x(X)=X^d+c_{d-1}X^{d-1}+\cdots+c_{1}X+c_{0}$ be the minimal polynomial of $x$ over $\mathbb{F}_q$.
Then $c_0=(-1)^d\Norm_{\mathbb{F}_q(x)/\mathbb{F}_q}(x)$ and $m_x(-1)=(-1)^d\Norm_{\mathbb{F}_q(x)/\mathbb{F}_q}(x+1)$.
\end{lem}

\begin{proof}
The $\mathbb{F}_q$-conjugates of $x$ are $x,x^q,\dots,x^{q^{d-1}}$, thus
$m_x(X)=\prod_{\nu=0}^{d-1}(X-x^{q^\nu})$.
Hence we calculate $c_0=m_x(0)=\prod_{\nu=0}^{d-1}(-x^{q^\nu})=(-1)^d\Norm_{\mathbb{F}_q(x)/\mathbb{F}_q}(x)$. Since the characteristic is $p$, we have $1+x^{q^\nu}=(1+x)^{q^\nu}$. Therefore,
$m_x(-1)=(-1)^d\Norm_{\mathbb{F}_q(x)/\mathbb{F}_q}(x+1)$.
\end{proof}

\begin{lem}\label{lem:kernel-of-norm0}
$C_0=\Ker(\Norm_{\mathbb{F}_{q^r}/\mathbb{F}_q})$.
\end{lem}

\begin{proof}
The restriction
$\Norm_{\mathbb{F}_{q^r}/\mathbb{F}_q}:\mathbb{F}_{q^r}^{\times}\to\mathbb{F}_q^{\times}$
is a surjective group homomorphism. Hence the homomorphism theorem yields $|\Ker(\Norm_{\mathbb{F}_{q^r}/\mathbb{F}_q})|=\frac{q^{r}-1}{q-1}=k$. On the other hand, $|\langle \omega^e\rangle|=\frac{q^{r}-1}{e}=k$. Since $\mathbb{F}_{q^r}^\times$ is cyclic, a subgroup of order $k$ is unique. Therefore $C_0=\Ker(\Norm_{\mathbb{F}_{q^r}/\mathbb{F}_q})$.
\end{proof}

\begin{lem}\label{lem:kernel-general-coset}
For $0\le a\le q-2$, we obtain
$C_a=\{x\in \mathbb{F}_{q^r}^\times\mid \Norm_{\mathbb{F}_{q^r}/\mathbb{F}_q}(x)=\omega^{ak}\}$.
\end{lem}

\begin{proof}
Note that $\Norm_{\mathbb{F}_{q^r}/\mathbb{F}_q}(\omega^a)=\omega^{ak}$.

\noindent$(\subset)$: Let $x\in C_a$. By the definition, we can denote $x=\omega^a z$ with $z\in C_0$. By Lemma \ref{lem:kernel-of-norm0}, we have $\Norm_{\mathbb{F}_{q^r}/\mathbb{F}_q}(z)=1$, and thus $\Norm_{\mathbb{F}_{q^r}/\mathbb{F}_q}(x)=\omega^{ak}$.

\noindent$(\supset)$: Take $x\in \mathbb{F}_{q^r}^\times$ so that $\Norm_{\mathbb{F}_{q^r}/\mathbb{F}_q}(x)=\omega^{ak}$. Then $\Norm_{\mathbb{F}_{q^r}/\mathbb{F}_q}(\omega^{-a}x)=1$. Hence, by Lemma \ref{lem:kernel-of-norm0}, we obtain $\omega^{-a}x\in C_0$. Therefore $x\in \omega^a C_0=C_a$.
\end{proof}

\begin{lem}\label{lem:q2-general-order1}
Assume that $q=2$. For every $r\in \mathbb{Z}_{\ge 2}$, we have $(0,0)_1=2^r-2$. Moreover, $(0,0)_1=\left\lceil \frac{k}{2}\right\rceil$ if $r=2$, and $(0,0)_1>\left\lceil \frac{k}{2}\right\rceil$ if $r\ge 3$.
\end{lem}

\begin{proof}
Since $\mathbb{F}_2^\times=\{1\}$, we have $a=b=0$. Denote $S:=\{x\in \mathbb{F}_{2^r}^\times\mid \Norm_{\mathbb{F}_{2^r}/\mathbb{F}_2}(x)=1,\ \Norm_{\mathbb{F}_{2^r}/\mathbb{F}_2}(x+1)=1\}$. By Lemma \ref{lem:kernel-general-coset}, we have $(0,0)_1=\#S$.

Now, since $\mathbb{F}_2^\times=\{1\}$, we have $\Norm_{\mathbb{F}_{2^r}/\mathbb{F}_2}(x)=1$ for every $x\in \mathbb{F}_{2^r}^\times$. Also, we have $\Norm_{\mathbb{F}_{2^r}/\mathbb{F}_2}(x+1)=1$ if and only if $x\neq 1$. Hence $S=\mathbb{F}_{2^r}^\times\setminus\{1\}$, and therefore $(0,0)_1=\#S=2^r-2$.
Since $k=2^r-1$, we have $\left\lceil \frac{k}{2}\right\rceil=2^{r-1}$. Thus $(0,0)_1=\left\lceil \frac{k}{2}\right\rceil$ if $r=2$, while $(0,0)_1>\left\lceil \frac{k}{2}\right\rceil$ if $r\ge 3$.
\end{proof}

\section{Estimates for $(a,b)_{q-1}\le \left\lceil \frac{k}{2}\right\rceil$}\label{sec:order-qminus1}

In this section, we prove the inequality $(a,b)_{q-1}\le \left\lceil \frac{k}{2}\right\rceil$. Throughout this section, we abbreviate $\Norm_{\mathbb{F}_{q^r}/\mathbb{F}_q}$ as $\Norm$.

\begin{prop}\label{prop:exact-character-formula-general}
Let $q=p^n$, $r\in \mathbb{Z}_{\ge 2}$, and $0\le a,b\le q-2$. Then
\begin{equation}\label{eq:exact-character-formula-general}
(a,b)_{q-1}=\frac{1}{(q-1)^2}
\sum_{\chi,\psi\in \widehat{\mathbb{F}_q^\times}}
\chi(\omega^{-ak})\psi(\omega^{-bk})
J_{\mathbb{F}_{q^r}}(\chi\circ \Norm,\psi\circ \Norm).
\end{equation}
\end{prop}

\begin{proof}
By Fact \ref{fac:standard-facts}, for each $x\in \mathbb{F}_{q^r}$ we have $(q-1)^2\one_{\Norm(x)=\omega^{ak}}\one_{\Norm(x+1)=\omega^{bk}}=\sum_{\chi,\psi\in\widehat{\mathbb{F}_q^\times}}\chi(\Norm(x)\omega^{-ak})\psi(\Norm(x+1)\omega^{-bk})$. Using Lemma \ref{lem:kernel-general-coset}, we obtain the formula.
\end{proof}

\begin{prop}\label{prop:main-term-explicit-general}
For $0\le a,b\le q-2$, define
\[
  A_r(a):=
  \begin{cases}
    -(q-2), & \textup{if } r \textup{ is even and } a=0,\\
    -(q-2), & \textup{if } r \textup{ is odd},\ q \textup{ is even, and } a=0,\\
    -(q-2), & \textup{if } r \textup{ is odd},\ q \textup{ is odd, and } a=\frac{q-1}{2},\\
    1, & \textup{otherwise,}
  \end{cases}\]
  \[
  B(b):=
  \begin{cases}
    -(q-2), & \textup{if } b=0,\\
    1, & \textup{if } b\neq 0,
  \end{cases}\quad
  C(a,b):=
  \begin{cases}
    -(q-2), & \textup{if } a=b,\\
    1, & \textup{if } a\neq b.
  \end{cases}
\]
Then
$(a,b)_{q-1}=M_{a,b}+E_{a,b}$,
where
$M_{a,b}:=\frac{q^r-2+A_r(a)+B(b)+C(a,b)}{(q-1)^2}$
and
$|E_{a,b}|\le \frac{(q-2)(q-3)}{(q-1)^2}q^{r/2}$.
\end{prop}

\begin{proof}
By \eqref{eq:exact-character-formula-general},
$(a,b)_{q-1}=\frac{1}{(q-1)^2}\sum_{\chi,\psi\in\widehat{\mathbb{F}_q^\times}}\chi(\omega^{-ak})\psi(\omega^{-bk})J_{\mathbb{F}_{q^r}}(\chi\circ \Norm,\psi\circ \Norm)$.
We divide this sum into five cases according to whether each $\chi$, $\psi$, and $\chi\psi$ is equal to $\one$.

\begin{enumerate}
\item[\textup{(1)}:]
If $\chi=\one$ and $\psi=\one$, then \eqref{eq:jacobi-norm-comparison-trivial} gives
$J_{\mathbb{F}_{q^r}}(\chi\circ \Norm,\psi\circ \Norm)=q^r-2$.

\item[\textup{(2)}:]
If $\chi\neq \one$ and $\psi=\one$, then $\chi\psi=\chi\neq \one$. By \eqref{eq:jacobi-special-right-trivial-J} and \eqref{eq:jacobi-norm-comparison-nontrivial}, we have
$J_{\mathbb{F}_{q^r}}(\chi\circ \Norm,\psi\circ \Norm)=(-1)^{r-1}J_{\mathbb{F}_q}(\chi,\psi)^r=-\chi((-1)^r)$. Hence
$\sum_{\chi\neq\one}
\chi(\omega^{-ak})
J_{\mathbb{F}_{q^r}}(\chi\circ\Norm,\one)
=
-\sum_{\chi\neq\one}\chi(\omega^{-ak})\chi((-1)^r)$.

\textup{(i):} Assume that $r$ is even. Then $\chi((-1)^r)=\chi(1)=1$. Hence, by Fact \ref{fac:standard-facts},
\[
-\sum_{\chi\neq\one}\chi(\omega^{-ak})\chi((-1)^r)
=
-\sum_{\chi\neq\one}\chi(\omega^{-ak})
=
\begin{cases}
-(q-2), & \textup{if } a=0,\\
1, & \textup{if } a\neq0.
\end{cases}
\]

\textup{(ii):} Assume that $r$ is odd and $q$ is even. Then $\chi((-1)^r)=\chi(-1)=\chi(1)=1$. Hence, by Fact \ref{fac:standard-facts},
\[
-\sum_{\chi\neq\one}\chi(\omega^{-ak})\chi((-1)^r)
=
-\sum_{\chi\neq\one}\chi(\omega^{-ak})
=
\begin{cases}
-(q-2), & \textup{if } a=0,\\
1, & \textup{if } a\neq0.
\end{cases}
\]

\textup{(iii):} Assume that $r$ is odd and $q$ is odd, then $\omega^{\frac{q-1}{2}k}=-1$. Therefore, by Fact \ref{fac:standard-facts},
\[
-\sum_{\chi\neq\one}\chi(\omega^{-ak})\chi((-1)^r)
=
-\sum_{\chi\neq\one}\chi(\omega^{(\frac{q-1}{2}-a)k})
=
\begin{cases}
-(q-2), & \textup{if } a=\frac{q-1}{2},\\
1, & \textup{if } a\neq\frac{q-1}{2}.
\end{cases}
\]
Thus this part gives $A_r(a)$.

\item[\textup{(3)}:]
If $\chi=\one$ and $\psi\neq \one$, then $\chi\psi=\psi\neq \one$. By \eqref{eq:jacobi-special-left-trivial} and \eqref{eq:jacobi-norm-comparison-nontrivial}, we have
$J_{\mathbb{F}_{q^r}}(\chi\circ \Norm,\psi\circ \Norm)=(-1)^{r-1}J_{\mathbb{F}_q}(\chi,\psi)^r=-1$.
Thus, by Fact \ref{fac:standard-facts},
\[
-\sum_{\psi\neq \one}\psi(\omega^{-bk})
=
\begin{cases}
-(q-2), & \textup{if } b=0,\\
1, & \textup{if } b\neq0.
\end{cases}
\]
Hence this part gives $B(b)$.

\item[\textup{(4)}:]
If $\chi\neq \one$, $\psi\neq \one$, and $\chi\psi=\one$, then $\psi=\chi^{-1}$. By \eqref{eq:jacobi-special-inverse-J} and \eqref{eq:jacobi-norm-comparison-nontrivial}, we have
$J_{\mathbb{F}_{q^r}}(\chi\circ \Norm,\psi\circ \Norm)=(-1)^{r-1}J_{\mathbb{F}_q}(\chi,\psi)^r=-1$.
Thus, by Fact \ref{fac:standard-facts},
\[
-\sum_{\chi\neq \one}\chi(\omega^{(b-a)k})
=
\begin{cases}
-(q-2), & \textup{if } a=b,\\
1, & \textup{if } a\neq b.
\end{cases}
\]
Hence this part gives $C(a,b)$.

\item[\textup{(5)}:]
Consider the case $\chi,\psi,\chi\psi\neq \one$. Define
$E_{a,b}:=\frac{1}{(q-1)^2}\sum_{\substack{\chi,\psi\in\widehat{\mathbb{F}_q^\times}\\ \chi,\psi,\chi\psi\neq \one}}\chi(\omega^{-ak})\psi(\omega^{-bk})J_{\mathbb{F}_{q^r}}(\chi\circ \Norm,\psi\circ \Norm)$.
By \eqref{eq:jacobi-absolute-value} and \eqref{eq:jacobi-norm-comparison-nontrivial}, we have
$|J_{\mathbb{F}_{q^r}}(\chi\circ \Norm,\psi\circ \Norm)|=q^{r/2}$.
Since there are $(q-2)(q-3)$ such pairs $(\chi,\psi)$, the triangle inequality gives
$|E_{a,b}|\le \frac{(q-2)(q-3)}{(q-1)^2}q^{r/2}$.
\end{enumerate}
Combining the five cases, we obtain
$(a,b)_{q-1}=\frac{q^r-2+A_r(a)+B(b)+C(a,b)}{(q-1)^2}+E_{a,b}=M_{a,b}+E_{a,b}$.
\end{proof}

\begin{cor}\label{cor:q-three-exact-general}
Assume that $q=3$. Then
\[
(a,b)_2=
\begin{cases}
\frac{3^r+1}{4}, & \textup{if } r \textup{ is odd},\ a=0,\textup{ and } b=1,\\
\frac{3^r-3}{4}, & \textup{if } r \textup{ is odd and either } a\neq0 \textup{ or } b\neq1,\\
\frac{3^r-5}{4}, & \textup{if } r \textup{ is even},\ a=0,\textup{ and } b=0,\\
\frac{3^r-1}{4}, & \textup{if } r \textup{ is even and either } a\neq0 \textup{ or } b\neq0.
\end{cases}
\]
In particular, $(a,b)_2\le \left\lceil \frac{k}{2}\right\rceil$ for every $0\le a,b\le 1$.
\end{cor}

\begin{proof}
By Proposition \ref{prop:main-term-explicit-general}, we have $|E_{a,b}|\le \frac{(q-2)(q-3)}{(q-1)^2}q^{r/2}$. Substituting $q=3$, we obtain $E_{a,b}=0$ for every choice of $a,b$. Our formula follows immediately from the definitions of $A_r(a)$, $B(b)$, and $C(a,b)$. The final assertion follows from \eqref{eq:ceil-half-k}.
\end{proof}

\begin{cor}\label{cor:character-sum-half-k-qgefour-rgeqthree}
Assume that $q\ge4$ and $r\ge3$. Then $(a,b)_{q-1}< \left\lceil \frac{k}{2}\right\rceil$ for every $0\le a,b\le q-2$.
\end{cor}

\begin{proof}
By Proposition \ref{prop:main-term-explicit-general}, we have
$(a,b)_{q-1}\le \frac{q^r+1}{(q-1)^2}+\frac{(q-2)(q-3)}{(q-1)^2}q^{r/2}$.

\noindent(i): Assume that $q$ is even or $r$ is odd. Then \eqref{eq:ceil-half-k} gives $\left\lceil \frac{k}{2}\right\rceil=\frac{q^r+q-2}{2(q-1)}$, and hence
\[
\left\lceil \frac{k}{2}\right\rceil-(a,b)_{q-1}
\ge
\frac{q-3}{2(q-1)^2}\left(q^r-2(q-2)q^{r/2}+q\right).
\]
Since $r\ge3$ and $q\ge4$, we have $q^{r/2}\ge q^{3/2}>2(q-2)$. Thus $q^r-2(q-2)q^{r/2}+q>0$, and the desired result follows.

\noindent(ii): Assume that $q$ is odd and $r$ is even. Then $q\ge5$ and $r\ge4$. By \eqref{eq:ceil-half-k}, we have $\left\lceil \frac{k}{2}\right\rceil=\frac{q^r-1}{2(q-1)}$, and hence
\[
\left\lceil \frac{k}{2}\right\rceil-(a,b)_{q-1}
\ge
\frac{q^{r/2}(q-3)\left(q^{r/2}-2(q-2)\right)-(q+1)}{2(q-1)^2}.
\]
Since $q^{r/2}\ge q^2$ and $q^{r/2}-2(q-2)\ge q^2-2q+4$, we obtain
$(q-3)\left(q^r-2(q-2)q^{r/2}\right)\ge (q-3)q^2(q^2-2q+4)>q+1$.
Therefore $(a,b)_{q-1}< \left\lceil \frac{k}{2}\right\rceil$.
\end{proof}

\subsection{$r=2$}\label{subsec:r2-general}
Next, we consider the case where $r=2$ separately. In this case, note that $k=\frac{q^2-1}{q-1}=q+1$.
\begin{prop}\label{prop:r2-discriminant-general}
Let $q$ be a power of an odd prime $p^n$,  $0\le a,b\le q-2$, and $\Delta(a,b):=(1+\omega^{ak}-\omega^{bk})^2-4\omega^{ak}$. Then we have
\[
  (a,b)_{q-1}=
  \begin{cases}
    1, & \textup{if } \Delta(a,b)=0,\\
    0, & \textup{if } \Delta(a,b)\neq 0 \textup{ and } \Delta(a,b) \textup{ is a square in } \mathbb{F}_{q},\\
    2, & \textup{if } \Delta(a,b) \textup{ is a nonsquare in } \mathbb{F}_{q}.
  \end{cases}
\]
\end{prop}

\begin{proof}
Let
$S:=\{x\in \mathbb{F}_{q^2}^\times\mid \Norm_{\mathbb{F}_{q^2}/\mathbb{F}_q}(x)=\omega^{ak},\ \Norm_{\mathbb{F}_{q^2}/\mathbb{F}_q}(x+1)=\omega^{bk}\}$. By Lemma \ref{lem:kernel-general-coset}, $(a,b)_{q-1}=\#S$. We decompose $S=(S\cap \mathbb{F}_{q})\sqcup(S\setminus \mathbb{F}_{q})$.

\noindent\textup{(1)}: Let $x\in S\cap \mathbb{F}_{q}$. Then $x^2=\omega^{ak}$ and $(x+1)^2=\omega^{bk}$. Subtracting these two equations gives $2x+1=\omega^{bk}-\omega^{ak}$. Since $q$ is a power of an odd prime, the element $2$ is invertible in $\mathbb{F}_{q}$, thus $x=\frac{\omega^{bk}-\omega^{ak}-1}{2}$ is uniquely determined. Substituting it into $x^2=\omega^{ak}$ yields $\Delta(a,b)=0$. Conversely, if $\Delta(a,b)=0$, then $x=\frac{\omega^{bk}-\omega^{ak}-1}{2}\in \mathbb{F}_{q}$ satisfies the two norm conditions above.

\noindent\textup{(2)}: Let $x\in S\setminus \mathbb{F}_{q}$. Since $\deg_{\mathbb{F}_{q}}(x)=2$, we have $\mathbb{F}_{q}(x)=\mathbb{F}_{q^2}$. Hence the minimal polynomial of $x$ over $\mathbb{F}_{q}$ is a monic polynomial of degree $2$ of the form $m_x(X)=X^2+c_{1}X+c_{0}$. By Lemma \ref{lem:minpoly-norm}, $c_{0}=\omega^{ak}$ and $1-c_{1}+\omega^{ak}=\omega^{bk}$. Therefore $c_{1}=1+\omega^{ak}-\omega^{bk}$, and thus $m_{x}(X)=X^2+(1+\omega^{ak}-\omega^{bk})X+\omega^{ak}$.

Conversely, if this polynomial is irreducible over $\mathbb{F}_{q}$, then its two roots lie in $\mathbb{F}_{q^2}$ and both have this polynomial as their minimal polynomial. Hence, by Lemma \ref{lem:minpoly-norm}, each root satisfies the required norm conditions. Thus the degree-$2$ contribution is exactly $2$ when this polynomial is irreducible over $\mathbb{F}_{q}$. On the other hand, if the polynomial splits over $\mathbb{F}_{q}$, then all its roots lie in $\mathbb{F}_{q}$, i.e., the degree-$2$ contribution is $0$.

The discriminant of this quadratic polynomial is $\Delta(a,b)=(1+\omega^{ak}-\omega^{bk})^2-4\omega^{ak}$. In odd characteristic, the cases $\Delta(a,b)=0$, $\Delta(a,b)\in (\mathbb{F}_q^\times)^2$, and $\Delta(a,b)\notin (\mathbb{F}_q^\times)^2$ correspond respectively to a double root, a split polynomial with two distinct roots, and an irreducible polynomial. Consequently, the corresponding contributions are $1$, $0$, and $2$, respectively.
\end{proof}

\begin{prop}\label{prop:r2-char2-general}
Assume that $q=2^n$, and let $0\le a,b\le q-2$. Put $c:=1+\omega^{ak}+\omega^{bk}$. Then
\[
(a,b)_{q-1}
=
\begin{cases}
1, & \textup{if } c=0,\\
0, & \textup{if } c\neq0 \textup{ and } \Tr_{\mathbb{F}_q/\mathbb{F}_2}(\omega^{ak}/c^2)=0,\\
2, & \textup{if } c\neq0 \textup{ and } \Tr_{\mathbb{F}_q/\mathbb{F}_2}(\omega^{ak}/c^2)=1.
\end{cases}
\]
\end{prop}

\begin{proof}
Let
$S:=\{x\in \mathbb{F}_{q^2}^\times\mid \Norm_{\mathbb{F}_{q^2}/\mathbb{F}_q}(x)=\omega^{ak},\ \Norm_{\mathbb{F}_{q^2}/\mathbb{F}_q}(x+1)=\omega^{bk}\}$. By Lemma \ref{lem:kernel-general-coset}, $(a,b)_{q-1}=\#S$. We decompose $S=(S\cap\mathbb{F}_q)\sqcup(S\setminus\mathbb{F}_q)$.

\noindent\textup{(1):} Let $x\in S\cap\mathbb{F}_q$. Then $x^2=\omega^{ak}$ and $(x+1)^2=\omega^{bk}$, thus $1+\omega^{ak}+\omega^{bk}=0$. Conversely, if $1+\omega^{ak}+\omega^{bk}=0$, then the unique element $x\in\mathbb{F}_q$ with $x^2=\omega^{ak}$ also satisfies $(x+1)^2=\omega^{bk}$. 

\noindent\textup{(2):} Let $x\in S\setminus\mathbb{F}_q$. By the same argument as in the proof of Proposition \ref{prop:r2-discriminant-general} \textup{(2)}, we have $m_{x}(X)=X^2+cX+\omega^{ak}$. If $c=0$, this polynomial has its unique root in $\mathbb{F}_q$. Hence there is no element $x\in S\setminus\mathbb{F}_q$ in this case. Assume that $c\neq0$. Changing variables $X=cY$, $X^2+cX+\omega^{ak}=c^2(Y^2+Y+\omega^{ak}/c^2)$. By Fact \ref{fac:artin-schreier-quadratic}, this polynomial is irreducible over $\mathbb{F}_q$ if and only if $\Tr_{\mathbb{F}_q/\mathbb{F}_2}(\omega^{ak}/c^2)=1$. If it is irreducible, then it has exactly two roots in $\mathbb{F}_{q^2}$, and both roots belong to $S\setminus\mathbb{F}_q$. Indeed, if $x$ is one of the roots, then the other root is $x^q$, and hence $\Norm_{\mathbb{F}_{q^2}/\mathbb{F}_q}(x)=\omega^{ak}$ and $\Norm_{\mathbb{F}_{q^2}/\mathbb{F}_q}(x+1)=(x+1)(x^q+1)=\omega^{ak}+c+1=\omega^{bk}$. If the polynomial is reducible, then its roots lie in $\mathbb{F}_q$. Hence there is no element $x\in S\setminus\mathbb{F}_q$ in this case. Combining \textup{(1)} and \textup{(2)}, we obtain the three cases in the statement.
\end{proof}

\begin{cor}\label{cor:r2-general-bound}
Let $q$ be a power of a prime $p^n$, and $0\le a,b\le q-2$. Then $(a,b)_{q-1}\in \{0,1,2\}$. In particular, if $q\ge 4$ and $r=2$, then $(a,b)_{q-1}< \left\lceil \frac{k}{2}\right\rceil$.
\end{cor}

\begin{proof}
$(a,b)_{q-1}\in \{0,1,2\}$ follows from Propositions \ref{prop:r2-discriminant-general} and \ref{prop:r2-char2-general}. Since $q\ge4$, we have $\left\lceil \frac{k}{2}\right\rceil=\left\lceil \frac{q+1}{2}\right\rceil\ge3$. 
\end{proof}

\begin{cor}\label{cor:character-sum-half-k-general}
Let $r\in \mathbb{Z}_{\ge2}$ and $0\le a,b\le q-2$. Then the following hold.
\begin{enumerate}
\item[\textup{(1)}] If $q=2$ and $r=2$, then $(0,0)_1=\left\lceil \frac{k}{2}\right\rceil$. If $q=2$ and $r\ge 3$, then $(0,0)_1>\left\lceil \frac{k}{2}\right\rceil$ 

\item[\textup{(2)}] If $q=3$, then $(a,b)_2\le \left\lceil \frac{k}{2}\right\rceil$. In particular, $(a,b)_2= \left\lceil \frac{k}{2}\right\rceil$ if and only if the following hold:
\[
\begin{cases}
r \textup{ is odd},\ a=0,\textup{ and } b=1,\\
r \textup{ is even},\textup{ and either } a\neq0 \textup{ or } b\neq0.
\end{cases}
\]

\item[\textup{(3)}] If $q\ge4$, then
$(a,b)_{q-1}< \left\lceil \frac{k}{2}\right\rceil$.
\end{enumerate}
Thus, $(a,b)_{q-1}\le \left\lceil \frac{k}{2}\right\rceil$ fails only in $q=2$ and $r\ge 3$.
\end{cor}

\begin{proof}
{ \,}\\
\noindent\textup{(1)}: It follows from Lemma \ref{lem:q2-general-order1}. 

\noindent\textup{(2)}: It follows from Corollary \ref{cor:q-three-exact-general}. 

\noindent\textup{(3)}: It follows from 
Corollaries \ref{cor:character-sum-half-k-qgefour-rgeqthree} and \ref{cor:r2-general-bound}.
\end{proof}

\begin{remark}\label{rem:q2-all-ones}
Assume that $q=2$. Then $e=q-1=1$, and hence $a=b=0$. By Lemma \ref{lem:q2-general-order1}, we have $(0,0)_1=2^r-2$. Since $k=2^r-1$, Lucas' theorem gives $\binom{k}{m}\equiv 1\pmod 2$ for every $0\le m\le k$. Moreover, $\omega^k=1$. Hence each entry of $C^{(0,0)}$ is $1$, and therefore $C^{(0,0)}$ is the all-one matrix.
\end{remark}

\section{A Cayley digraph interpretation}\label{sec:cayley-digraph}

In this section, we give an exact formula for $(a,b)_{q-1}$ in terms of directed walks in a Cayley digraph. It gives a structural interpretation which is complementary to the estimate in \S\ref{sec:order-qminus1}. Throughout this section, we continuously abbreviate $\Norm_{\mathbb{F}_{q^r}/\mathbb{F}_q}$ as $\Norm$.

\begin{defi}\label{def:Wuv-cayley}
{\,}
\begin{itemize}
\item For $u,v\in \mathbb{F}_q^\times$, define $W_{u,v}(q,r):=\#\{(x_1,\dots,x_r)\in (\mathbb{F}_q\setminus\{0,1\})^r \mid \prod_{t=1}^r x_t=u,\ \prod_{t=1}^r (1-x_t)=v\}$.
\item Define $G_q:=\mathbb{F}_q^\times\times \mathbb{F}_q^\times$ and $\Omega_q:=\{(x,1-x)\mid x\in \mathbb{F}_q\setminus\{0,1\}\}\subset G_q$.
\item Define $\Gamma_q$ to be the directed Cayley graph with vertex set $G_q$, whose directed edges from $g$ to $gh$ are $(g,gh)$ for $g\in G_q$ and $h\in \Omega_q$.
\item Fix an ordering of $G_q$. For $(u,v)\in G_q$, let $\nu(u,v)$ denote the index of the vertex $(u,v)$. Let $A_q$ be the adjacency matrix of $\Gamma_q$ whose $(\nu(u,v),\nu(ux,v(1-x)))$-entry is $1$ for $(u,v)\in G_q$ and $x\in\mathbb{F}_q\setminus\{0,1\}$.
\end{itemize}
\end{defi}

\begin{lem}\label{lem:Wuv-adjacency}
For every $u,v\in \mathbb{F}_q^\times$, we have $W_{u,v}(q,r)=(A_q^r)_{\nu(1,1),\nu(u,v)}$.
\end{lem}

\begin{proof}
A $r$-length walk in $\Gamma_q$ starting from $(1,1)$ is determined by a sequence $x_1,\dots,x_r\in \mathbb{F}_q\setminus\{0,1\}$, and its terminal vertex is $(\prod_{t=1}^r x_t,\ \prod_{t=1}^r (1-x_t))$. Thus the number of such walks ending at $(u,v)$ is exactly $W_{u,v}(q,r)$.
\end{proof}

\begin{prop}\label{prop:general-order-qminus1-graph}
Let $0\le a,b\le q-2$. Then
\begin{subequations}
\begin{align}
(a,b)_{q-1}
&=
\frac{q^r-(-1)^{r-1}(q-2)^r-2}{(q-1)^2}
+(-1)^{r-1}W_{(-1)^r\omega^{ak},\omega^{bk}}(q,r),
\label{eq:general-order-qminus1-W-formula}\\
(a,b)_{q-1}
&=
\frac{q^r-(-1)^{r-1}(q-2)^r-2}{(q-1)^2}
+(-1)^{r-1}(A_q^r)_{\nu(1,1),\nu((-1)^r\omega^{ak},\omega^{bk})}.
\label{eq:general-order-qminus1-adjacency-formula}
\end{align}
\end{subequations}
\end{prop}

\begin{proof}
By the definition of $W_{u,v}(q,r)$ and Fact \ref{fac:standard-facts}, we have
\[
(q-1)^2W_{u,v}(q,r)
=
\sum_{\chi,\psi\in \widehat{\mathbb{F}_q^\times}}
\chi(u^{-1})\psi(v^{-1})
\left(\sum_{x\in\mathbb{F}_q\setminus\{0,1\}}\chi(x)\psi(1-x)\right)^r\\
=
\sum_{\chi,\psi\in \widehat{\mathbb{F}_q^\times}}
\chi(u^{-1})\psi(v^{-1})J'_{\mathbb{F}_q}(\chi,\psi)^r.\]

Since $\chi(-1)^2=1$, \eqref{eq:J-Jprime-relation} gives
$J'_{\mathbb{F}_q}(\chi,\psi)^r=\chi((-1)^r)J_{\mathbb{F}_q}(\chi,\psi)^r$.
Thus
\begin{align*}
(q-1)^2W_{u,v}(q,r)
&=
\sum_{\chi,\psi\in \widehat{\mathbb{F}_q^\times}}
\chi((-1)^r u^{-1})\psi(v^{-1})J_{\mathbb{F}_q}(\chi,\psi)^r.
\end{align*}
Substituting $u=(-1)^r\omega^{ak}$ and $v=\omega^{bk}$, we obtain
\begin{equation}\label{eq:specialized-W-character-formula}
(q-1)^2W_{(-1)^r\omega^{ak},\omega^{bk}}(q,r)
=
\sum_{\chi,\psi\in \widehat{\mathbb{F}_q^\times}}
\chi(\omega^{-ak})\psi(\omega^{-bk})J_{\mathbb{F}_q}(\chi,\psi)^r.
\end{equation}

By \eqref{eq:weighted-jacobi-norm-comparison} with $u=\omega^{ak}$ and $v=\omega^{bk}$, together with \eqref{eq:specialized-W-character-formula}, we obtain
\begin{equation}\label{eq:specialized-norm-character-formula}
\sum_{\chi,\psi\in \widehat{\mathbb{F}_q^\times}}
\chi(\omega^{-ak})\psi(\omega^{-bk})
J_{\mathbb{F}_{q^r}}(\chi\circ\Norm,\psi\circ\Norm)
=
q^r-2+(-1)^{r-1}\bigl((q-1)^2W_{(-1)^r\omega^{ak},\omega^{bk}}(q,r)-(q-2)^r\bigr).
\end{equation}
On the other hand, \eqref{eq:exact-character-formula-general} gives
$(q-1)^2(a,b)_{q-1}=\sum_{\chi,\psi\in \widehat{\mathbb{F}_q^\times}}\chi(\omega^{-ak})\psi(\omega^{-bk})J_{\mathbb{F}_{q^r}}(\chi\circ\Norm,\psi\circ\Norm)$.
Combining it with \eqref{eq:specialized-norm-character-formula}, we obtain
\[
(q-1)^2(a,b)_{q-1}
=
q^r-2+(-1)^{r-1}\bigl((q-1)^2W_{(-1)^r\omega^{ak},\omega^{bk}}(q,r)-(q-2)^r\bigr).
\]
Dividing by $(q-1)^2$, we obtain \eqref{eq:general-order-qminus1-W-formula}. Finally, \eqref{eq:general-order-qminus1-adjacency-formula} follows from Lemma \ref{lem:Wuv-adjacency}.
\end{proof}

\begin{exam}\label{exam:q5-adjacency-general}
Take $q=5$. For $s,t\in\{1,2,3,4\}$, put
$w_{4(s-1)+t}:=(s,t)$ and $\nu(s,t):=4(s-1)+t$.
The adjacency matrix of $\Gamma_5$ with respect to this numbering is
\[
A_5=
\begin{psmallmatrix}
0&0&0&0&0&0&0&1&0&0&1&0&0&1&0&0\\
0&0&0&0&0&0&1&0&1&0&0&0&0&0&0&1\\
0&0&0&0&0&1&0&0&0&0&0&1&1&0&0&0\\
0&0&0&0&1&0&0&0&0&1&0&0&0&0&1&0\\
0&0&1&0&0&0&0&0&0&1&0&0&0&0&0&1\\
1&0&0&0&0&0&0&0&0&0&0&1&0&0&1&0\\
0&0&0&1&0&0&0&0&1&0&0&0&0&1&0&0
\\
0&1&0&0&0&0&0&0&0&0&1&0&1&0&0&0\\
0&0&0&1&0&1&0&0&0&0&0&0&0&0&1&0\\
0&0&1&0&0&0&0&1&0&0&0&0&1&0&0&0\\
0&1&0&0&1&0&0&0&0&0&0&0&0&0&0&1\\
1&0&0&0&0&0&1&0&0&0&0&0&0&1&0&0\\
0&1&0&0&0&0&1&0&0&0&0&1&0&0&0&0\\
0&0&0&1&1&0&0&0&0&0&1&0&0&0&0&0\\
1&0&0&0&0&0&0&1&0&1&0&0&0&0&0&0\\
0&0&1&0&0&1&0&0&1&0&0&0&0&0&0&0
\end{psmallmatrix}.
\]
Choose $\omega$ so that $\omega^k=2$ in $\mathbb{F}_5^\times$. If $r=6$, then $k=(5^6-1)/(5-1)$, and Proposition \ref{prop:general-order-qminus1-graph} gives $(a,b)_4=1022-(A_5^6)_{\nu(1,1),\nu(2^a,2^b)}$. In particular, $(0,0)_4=1022-(A_5^6)_{1,1}=1022-90=932$, and $(0,1)_4=1022-(A_5^6)_{1,2}=1022-42=980$.
\end{exam}

\begin{remark}\label{rem:half-k-adjacency-criterion}
Combining \eqref{eq:ceil-half-k} with \eqref{eq:general-order-qminus1-adjacency-formula} gives the following equivalent conditions for $(a,b)_{q-1}\le \left\lceil \frac{k}{2}\right\rceil$.
\begin{enumerate}
\item[\textup{(1)}]
If $r$ is odd, then
$(A_q^r)_{\nu(1,1),\nu(-\omega^{ak},\omega^{bk})}\le \frac{(q-3)q^r+2(q-2)^r+q^2-3q+6}{2(q-1)^2}$.

\item[\textup{(2)}]
If $r$ is even and $q$ is odd, then 
$(A_q^r)_{\nu(1,1),\nu(\omega^{ak},\omega^{bk})}\ge \frac{-(q-3)q^r+2(q-2)^r+q-5}{2(q-1)^2}$.

\item[\textup{(3)}]
If $r$ is even and $q$ is even, then 
$(A_q^r)_{\nu(1,1),\nu(\omega^{ak},\omega^{bk})}\ge \frac{-(q-3)q^r+2(q-2)^r-q^2+3q-6}{2(q-1)^2}$.
\end{enumerate}
In particular, when $r$ is even and $q\neq2$, these conditions are automatically satisfied since the right-hand sides in \textup{(2)} and \textup{(3)} are negative, while the left-hand sides are walk counts.
\end{remark}

\section{Sharper bounds in odd prime $r$ }\label{sec:small-r}

In this section, we consider the case where $r$ is an odd prime $\ell$. Particularly,  we consider the case $r=3$ separately.

Define
$T_{\ell}(a,b):=\{x\in \mathbb{F}_q\mid x^\ell=\omega^{ak},\ (x+1)^\ell=\omega^{bk}\}$,
and
$I_{\ell}(a,b):=\{m_x(X)\in \mathbb{F}_q[X]\mid x\in \bar{\mathbb{F}}_q,\ [\mathbb{F}_q(x):\mathbb{F}_q]=\ell,\ m_x(0)=-\omega^{ak},\ m_x(-1)=-\omega^{bk}\}$.

\begin{prop}\label{prop:odd-prime-structure-general}
Let $r$ be an odd prime $\ell$, $q$ be a prime power $p^n$, and $0\le a,b\le q-2$. Then $(a,b)_{q-1}=\#T_{\ell}(a,b)+\ell \#I_{\ell}(a,b)$.
\end{prop}

\begin{proof}
Put $S:=\{x\in \mathbb{F}_{q^\ell}^\times\mid \Norm_{\mathbb{F}_{q^\ell}/\mathbb{F}_q}(x)=\omega^{ak},\ \Norm_{\mathbb{F}_{q^\ell}/\mathbb{F}_q}(x+1)=\omega^{bk}\}$. By Lemma \ref{lem:kernel-general-coset}, we have $(a,b)_{q-1}=\#S$. Since $[\mathbb{F}_{q^\ell}:\mathbb{F}_q]=\ell$ is prime, we decompose $S=(S\cap \mathbb{F}_q)\sqcup(S\setminus \mathbb{F}_q)$.

\noindent\textup{(1)}: Let $x\in S\cap \mathbb{F}_q$. Then $\Norm_{\mathbb{F}_{q^\ell}/\mathbb{F}_q}(x)=x^\ell$ and $\Norm_{\mathbb{F}_{q^\ell}/\mathbb{F}_q}(x+1)=(x+1)^\ell$. Therefore $\#(S\cap \mathbb{F}_q)=\#T_{\ell}(a,b)$.

\noindent\textup{(2)}: Let $x\in S\setminus \mathbb{F}_q$. Define $\Phi:S\setminus \mathbb{F}_q\to I_{\ell}(a,b)$ $(x\mapsto m_x(X))$. Since $x\notin \mathbb{F}_q$, we have $\mathbb{F}_q(x)=\mathbb{F}_{q^\ell}$. Hence $m_x(X)$ is a monic irreducible polynomial over $\mathbb{F}_q$ of degree $\ell$. Moreover, Lemma \ref{lem:minpoly-norm} gives $m_x(0)=-\omega^{ak}$ and $m_x(-1)=-\omega^{bk}$. Thus $m_x(X)\in I_{\ell}(a,b)$.

Now let $f(X)\in I_{\ell}(a,b)$. Since $f(X)$ is irreducible of degree $\ell$, it has exactly $\ell$ distinct roots in $\mathbb{F}_{q^\ell}$. Let $y$ be one of these roots. Then $[\mathbb{F}_q(y):\mathbb{F}_q]=\ell$, hence $\mathbb{F}_q(y)=\mathbb{F}_{q^\ell}$. Moreover, $m_y(X)=f(X)$, and Lemma \ref{lem:minpoly-norm} gives $\Norm_{\mathbb{F}_{q^\ell}/\mathbb{F}_q}(y)=\omega^{ak}$ and $\Norm_{\mathbb{F}_{q^\ell}/\mathbb{F}_q}(y+1)=\omega^{bk}$. Thus $y\in S\setminus \mathbb{F}_q$. Therefore $\Phi^{-1}(f)$ exactly consists of the $\ell$ roots of $f$, and hence $\#\Phi^{-1}(f)=\ell$.

It follows that $\#(S\setminus \mathbb{F}_q)=\ell\#I_{\ell}(a,b)$. Combining (1) and (2), we obtain the statement.
\end{proof}

\begin{prop}\label{lem:T-bound-odd-prime-general}
Let $r$ be an odd prime $\ell$, $q$ be a prime power $p^n$, and $0\le a,b\le q-2$. Then
\[
\begin{aligned}
\#T_{\ell}(a,b)&=
\begin{cases}
1, & \textup{if }p=\ell,  1+\omega^{ak}-\omega^{bk}=0,\\
0, & \textup{if }p=\ell,  1+\omega^{ak}-\omega^{bk}\neq0,
\end{cases}
&
0\le \#T_{\ell}(a,b)&\le \ell-1
\quad \textup{if }p\neq\ell.
\end{aligned}
\]
\end{prop}

\begin{proof}
We prove the statement by considering the following two cases.

\noindent\textup{(1)}: Assume that $p=\ell$. Then $(X+1)^\ell=X^\ell+1$. If $1+\omega^{ak}-\omega^{bk}\neq0$, then $T_{\ell}(a,b)=\emptyset$. If $1+\omega^{ak}-\omega^{bk}=0$, then $T_{\ell}(a,b)=\{x\in\mathbb{F}_q\mid x^\ell=\omega^{ak}\}$. Since $x\mapsto x^\ell$ is the Frobenius automorphism of $\mathbb{F}_q$, this set has exactly one element. It proves our formula in the case $p=\ell$.

\noindent\textup{(2)}: Assume that $p\neq\ell$. If $x\in T_{\ell}(a,b)$, then $x$ is a root of $(X+1)^\ell-X^\ell+\omega^{ak}-\omega^{bk}$. Since the highest degree term of this polynomial is $\ell X^{\ell-1}$, it has degree $\ell-1$. Therefore $0\le \#T_{\ell}(a,b)\le \ell-1$.
\end{proof}

Define
$P_{\ell}(a,b):=\{f(X)\in \mathbb{F}_q[X]\mid \deg f=\ell,\ f \textup{ is monic},\ f(0)=-\omega^{ak},\ f(-1)=-\omega^{bk}\}$.
Then $I_{\ell}(a,b)\subset P_{\ell}(a,b)$.

\begin{prop}\label{lem:I-bound-odd-prime-general}
Let $r$ be an odd prime $\ell$, $q$ be a prime power $p^n$, and $0\le a,b\le q-2$. Then $0\le \#I_{\ell}(a,b)\le q^{\ell-2}$.
\end{prop}

\begin{proof}
By the definition, $\#I_{\ell}(a,b)\ge 0$. Any $f(X)\in P_{\ell}(a,b)$ can be written in the form $f(X)=X^\ell+c_{\ell-1}X^{\ell-1}+\cdots+c_1X-\omega^{ak}$. Since $\ell$ is an odd prime, the condition $f(-1)=-\omega^{bk}$ gives $c_1=\sum_{\mu=2}^{\ell-1}(-1)^\mu c_\mu-(1+\omega^{ak}-\omega^{bk})$. Therefore an element of $P_{\ell}(a,b)$ is determined by $\ell-2$ coefficients, and hence $\#P_{\ell}(a,b)=q^{\ell-2}$. Since $I_{\ell}(a,b)\subset P_{\ell}(a,b)$, it follows that $\#I_{\ell}(a,b)\le q^{\ell-2}$.
\end{proof}

\begin{cor}\label{cor:V-bound-odd-prime-general}
Let $r$ be an odd prime $\ell$, $q$ be a prime power $p^n$, and $0\le a,b\le q-2$. Then
\[
0\le (a,b)_{q-1}\le
\begin{cases}
\ell q^{\ell-2}+1, & \textup{if } p=\ell,\\
\ell q^{\ell-2}+\ell-1, & \textup{if } p\neq\ell.
\end{cases}
\]
\end{cor}

\begin{proof}
It follows immediately from Propositions \ref{prop:odd-prime-structure-general}, \ref{lem:T-bound-odd-prime-general}, and \ref{lem:I-bound-odd-prime-general}.
\end{proof}

Especially, we can obtain a concrete value of our bound when $r=3$. Thus we consider $r=3$.

\begin{lem}\label{lem:resultant-r3-general}
Let $r=3$, $q$ be a prime power $p^n$, and $0\le a,b\le q-2$. Put $A(X):=X^3-\omega^{ak}$ and $B(X):=(X+1)^3-\omega^{bk}$. Then $\Res(A,B)=(1+\omega^{ak}-\omega^{bk})^3+27\omega^{ak}\omega^{bk}$.
\end{lem}

\begin{proof}
Put $c:=1+\omega^{ak}-\omega^{bk}$. By the definition of the Sylvester matrix, we have
$S(A,B)=\begin{psmallmatrix} A_{11} & A_{12}\\ A_{21} & A_{22} \end{psmallmatrix}$, where
$A_{11}=I_3$, $A_{12}=-\omega^{ak}I_3$,
$A_{21}=\begin{psmallmatrix}1&3&3\\0&1&3\\0&0&1\end{psmallmatrix}$, and
$A_{22}=\begin{psmallmatrix}1-\omega^{bk}&0&0\\3&1-\omega^{bk}&0\\3&3&1-\omega^{bk}\end{psmallmatrix}$.
Put
$L:=A_{21}-A_{11}=\begin{psmallmatrix}0&3&3\\0&0&3\\0&0&0\end{psmallmatrix}$
and
$M:=A_{22}-A_{12}=\begin{psmallmatrix}c&0&0\\3&c&0\\3&3&c\end{psmallmatrix}$.
Then $\Res(A,B)=\det\begin{psmallmatrix} I_3 & -\omega^{ak}I_3\\ L & M \end{psmallmatrix}=\det(M+\omega^{ak}L)=c^3-27c\omega^{ak}+27\omega^{2ak}+27\omega^{ak}=(1+\omega^{ak}-\omega^{bk})^3+27\omega^{ak}\omega^{bk}$.
\end{proof}

Note that
$T_{3}(a,b):=\{x\in \mathbb{F}_q\mid x^3=\omega^{ak},\ (x+1)^3=\omega^{bk}\}$,
and
$I_{3}(a,b):=\{m_x(X)\in \mathbb{F}_q[X]\mid x\in \bar{\mathbb{F}}_q,\ [\mathbb{F}_q(x):\mathbb{F}_q]=3,\ m_x(0)=-\omega^{ak},\ m_x(-1)=-\omega^{bk}\}$.

\begin{prop}\label{lem:T-layer-r3-general}
Let $r=3$, $q$ be a prime power $p^n$, and $0\le a,b\le q-2$.
Put $A(X):=X^{3}-\omega^{ak}$ and $B(X):=(X+1)^3-\omega^{bk}$.
Let $\lambda:=0$ if $p=2$, and let $\lambda:=\frac{q-1}{2}$ if $p$ is odd.
Then
\[
\#T_3(a,b)=
\begin{cases}
1, & \textup{if } p=3 \textup{ and } 1+\omega^{ak}-\omega^{bk}=0,\\
0, & \textup{if } p=3 \textup{ and } 1+\omega^{ak}-\omega^{bk}\neq0,\\
2, & \textup{if } p\neq3,\ (a,b)=(0,\lambda),\ 3\mid(q-1),\\
0, & \textup{if } p\neq3,\ (a,b)=(0,\lambda),\ 3\nmid(q-1),\\
1, & \textup{if } p\neq3,\ (a,b)\neq(0,\lambda),\ \Res(A,B)=0,\\
0, & \textup{if } p\neq3,\ (a,b)\neq(0,\lambda),\ \Res(A,B)\neq0.
\end{cases}
\]
\end{prop}

\begin{proof}
We divide the argument into two cases: \textup{(1)}: $p=3$, and \textup{(2)}: $p\neq 3$.

\noindent\textup{(1)}: Suppose $p=3$. Then $B(X)-A(X)=1+\omega^{ak}-\omega^{bk}$. Hence a common root exists if and only if $1+\omega^{ak}-\omega^{bk}=0$. If the condition holds, then the equation $x^3=\omega^{ak}$ has exactly one solution in $\mathbb{F}_q$, since the Frobenius map $x\mapsto x^3$ is an automorphism of $\mathbb{F}_q$. Therefore $\#T_3(a,b)=1$. If $1+\omega^{ak}-\omega^{bk}\neq 0$, then there is no common root and $\#T_3(a,b)=0$.

\noindent\textup{(2)}: Suppose $p\neq 3$. Put $u:=\omega^{ak}$ and $v:=\omega^{bk}$. Then $3A(X)=(X-1)(B(X)-A(X))+L(X)$, where $L(X):=(v-u+2)X+(1-2u-v)$. Let $D(X):=\gcd(A(X),B(X))$. Since $D(X)$ divides both $A(X)$ and $B(X)$, it divides $B(X)-A(X)$. Hence $D(X)\mid L(X)$. The equality $L(X)=0$ holds if and only if $(u,v)=(1,-1)$. By the definition of $\lambda$, it is equivalent to $(a,b)=(0,\lambda)$.

\textup{(i)}: Assume that $(a,b)\neq(0,\lambda)$. Then $L(X)$ is not the zero polynomial. Since $D(X)\mid L(X)$ and $\deg L\le1$, we have $\deg D(X)\le1$. On the other hand, Lemma \ref{lem:resultant-r3-general} gives $\Res(A,B)=(1+u-v)^3+27uv$. By Fact \ref{fac:resultant-common-root}, the condition $\Res(A,B)=0$ is equivalent to the existence of a nonconstant common divisor of $A(X)$ and $B(X)$. Therefore $D(X)$ has degree $1$ if $\Res(A,B)=0$, and then there is exactly one common root in $\mathbb{F}_q$. If $\Res(A,B)\neq0$, then there is no common root.

\textup{(ii)}: Assume that $(a,b)=(0,\lambda)$. Then $A(X)=X^3-1$ and $B(X)=(X+1)^3+1$, and hence $B(X)-A(X)=3(X^2+X+1)$. Since $p\neq3$, we have $D(X)=X^2+X+1$. The roots of $X^2+X+1$ are primitive cubic roots of unity, and thus $X^2+X+1$ splits over $\mathbb{F}_q$ if and only if $3\mid(q-1)$. Therefore $\#T_3(a,b)=2$ if $3\mid(q-1)$, while $\#T_3(a,b)=0$ if $3\nmid(q-1)$.
\end{proof}

\begin{prop}\label{lem:I-layer-r3-general}
Let $r=3$, $q$ be a prime power $p^n$, and $0\le a,b\le q-2$. Define $\phi_{a,b}:\mathbb{F}_q\setminus\{0,-1\}\to\mathbb{F}_q\ (x\mapsto x-1-\frac{\omega^{ak}}{x}+\frac{\omega^{bk}}{x+1})$. Then $\#I_{3}(a,b)=q-\#\mathrm{Im}(\phi_{a,b})$. In particular, $2\le \#I_{3}(a,b)\le \lfloor\frac{2q+2}{3}\rfloor$.
\end{prop}

\begin{proof}
Any polynomial $m_x(X)\in I_{3}(a,b)$ can be written uniquely in the form $m_x(X)=X^3+c_2X^2+c_1X+c_0$. The condition $m_x(0)=-\omega^{ak}$ gives $c_0=-\omega^{ak}$, while $m_x(-1)=-\omega^{bk}$ yields $c_2-c_1=1+\omega^{ak}-\omega^{bk}$. Writing $c_2=-s$, we obtain
$m_x(X)=X^3-sX^2-(s+1+\omega^{ak}-\omega^{bk})X-\omega^{ak}$.
Therefore $\#I_3(a,b)$ is exactly the number of parameters $s\in\mathbb{F}_q$ for which this polynomial is irreducible over $\mathbb{F}_q$.

Since $m_x(0)=-\omega^{ak}\neq 0$ and $m_x(-1)=-\omega^{bk}\neq 0$, this polynomial is reducible over $\mathbb{F}_q$ if and only if it has a root in $\mathbb{F}_q\setminus\{0,-1\}$. For $x\in \mathbb{F}_q\setminus\{0,-1\}$, substituting $x$ into the above polynomial gives
$x^3-sx^2-(s+1+\omega^{ak}-\omega^{bk})x-\omega^{ak}=0$.
Since $x^2+x=x(x+1)\neq0$, this equation is equivalent to
$s=x-1-\frac{\omega^{ak}}{x}+\frac{\omega^{bk}}{x+1}=\phi_{a,b}(x)$.
Therefore the above polynomial is reducible over $\mathbb{F}_q$ if and only if $s\in \mathrm{Im}(\phi_{a,b})$, and hence $\#I_3(a,b)=q-\#\mathrm{Im}(\phi_{a,b})$.

Since $|\mathbb{F}_{q}\backslash \{0,-1\}|=q-2$, we have $\#\mathrm{Im}(\phi_{a,b})\le q-2$, and therefore $\#I_3(a,b)\ge 2$. On the other hand, for each $s\in \mathbb{F}_q$, the fiber $\phi_{a,b}^{-1}(s)$ consists of roots of the above cubic polynomial, and hence it has a size of at most $3$. Therefore $q-2\le 3\#\mathrm{Im}(\phi_{a,b})$, and thus $\#I_3(a,b)\le q-\lceil\frac{q-2}{3}\rceil=\lfloor\frac{2q+2}{3}\rfloor$.
\end{proof}

\begin{cor}\label{cor:r3-structure-general}
Let $r=3$, $q$ be a prime power $p^n$, and for every $0\le a,b\le q-2$. Then we have
$\# T_{3}(a,b)+6\le (a,b)_{q-1}\le \# T_{3}(a,b)+3\lfloor\frac{2q+2}{3}\rfloor$.
In particular, we have $6\le (a,b)_{q-1}\le 2q+4$.
\end{cor}

\begin{proof}
Proposition \ref{prop:odd-prime-structure-general} with $\ell=3$ gives $(a,b)_{q-1}=\#T_3(a,b)+3\#I_3(a,b)$. Hence the bounds follow from Propositions \ref{lem:T-layer-r3-general} and \ref{lem:I-layer-r3-general}. Moreover, Proposition \ref{lem:T-layer-r3-general} gives $\#T_3(a,b)\le 2$, and hence $(a,b)_{q-1}\le 2+3\lfloor(2q+2)/3\rfloor\le 2q+4$.
\end{proof}

\paragraph{Question.}
In our family, the only failures of the inequality $(a,b)_{q-1}\le \lceil \frac{k}{2}\rceil$ occur when $q=2$ and $r\ge 3$, where the corresponding matrix $C^{(0,0)}$ is the all-one matrix. It suggests the following problem.
\begin{itemize}[leftmargin=2em]
  \item In more general cases, when does the inequality $(a,b)_e\le \lceil \frac{k}{2}\rceil$ fail? 
  \item What is the concrete shape of the corresponding matrix $C^{(a,b)}$ in the failure case?
\end{itemize}

\section*{Acknowledgements}

This work was supported by JST SPRING, Grant Number JPMJSP2152. We thank Professor Koichi Betsumiya for his guidance and encouragement throughout this research. We also thank Wei-Liang Sun for helpful comments and valuable information on cyclotomic numbers.

\end{document}